\documentclass[final,reqno,twoside,a4paper,11pt]{amsart}
\usepackage{mltools}
\usepackage{mlmath62,mlthm10-REAR,mllocal}
\usepackage{upref,amsmath,amssymb,amsthm,mathrsfs}
\usepackage{tikz,graphics,latexsym}
\usepackage{times}
\usepackage{amsmath}
\usepackage{amssymb}
\usepackage{amsthm}
\usepackage{verbatim}
\usepackage{amsmath}
\usepackage{amsfonts}
\usepackage{amssymb}
\usepackage[all]{xy}
\usepackage{xcolor}
\usepackage[utf8]{inputenc}

\usepackage[colorlinks=true,linkcolor=blue,citecolor=blue]{hyperref}

\usepackage[scaled]{helvet} 
\usepackage{courier} 
\usepackage[mathbf]{euler}
\usepackage{mllocal} 

\usepackage{pgfplots}

\numberwithin{equation}{section}

\begin{document}


\title[Stratified Morse critical points and  locally tame singularities]{Stratified Morse critical points and Brasselet number on non-degenerate locally tame singularities}

\author{Tha\'is M. Dalbelo}
\address{Departamento de Matem\'atica, 
	Universidade Federal de S\~ao Carlos (UFSCar),
	Brazil}
\email{thaisdalbelo@ufscar.br}

\author{Hellen Santana}
\address{Instituto de Ci\^encias Tecnológicas e Exatas, Universidade Federal do Triângulo Mineiro (UFTM),	Brazil}
\email{hellen.santana@uftm.edu.br}

\thanks{Tha\'is M. Dalbelo is partially support by FAPESP under grants $2019/21181-0$ and $2023/01018-2$, Hellen Santana is partially support by FAPESP under grants $2019/21181-0$ and $2022/06968-6$.}


\begin{abstract}

The generalization of the Morse theory presented by Goresky and MacPherson is 
a landmark that divided completely the topological and geo\-me\-tri\-cal 
study of singular spaces. Let \{$X_t\}_t$ be a suitable family of germs at $0$ of complete intersection varieties in  $\mathbb{C}^n$ and $\{f_t\}_t, \{g_t\}_t$ families of non-constant polynomial functions on $X_t$. If the germs $X_t$, $X_t \cap f_t^{-1}(0)$ and $X_t\cap f_t^{-1}(0) \cap g_t^{-1}(0)$ are non-degenerate, locally tame, complete intersection varieties, for each $t,$ we prove that the difference of the Brasselet numbers,  
${\rm B}_{f_t,X_t}(0)$ and ${\rm B}_{f_t,X_t\cap g_t^{-1}(0)}(0)$, is related with the 
number of Morse critical points {on the regular part of the Milnor fiber} of 
$f_t$ appearing in a morsefication of $g_t$, even in the case where $g_t$ has a 
critical locus with arbitrary dimension. This result connects topological and 
geometric properties and allows us to determine some interesting formulae, 
mainly in terms of the combinatorial information from Newton polyhedra.  
\end{abstract}

\maketitle

\tableofcontents

\section{Introduction}

Given a topological space $X$, a smooth real valued function $f$ on $X$ and  a real number $c$, the fundamental problem of Morse theory is to study the topological changes in the space $X_c = \{f \leq c\}$ as the number $c$ varies.

In classical Morse Theory, the space $X$ is taken to be a compact differentiable manifold. In \cite{GM}, Goresky and MacPherson extended Morse Theory to the setting of Whitney stratified spaces.

An important invariant associated to a germ of an analytic function $f:(\mathbb{C}^n,0) 
\to (\mathbb{C},0)$ with an isolated critical point at the origin is denoted by $\mu(f)$ and it is defined as 
\[
\mu(f):=dim_{\mathbb{C}}\frac{\mathcal{O}_n}{J(f )},
\] where $\mathcal{O}_n$ is the ring of germs of analytic functions at the 
origin, and $J(f)$ is the Jacobian ideal of $f$. This invariant, defined by Milnor in \cite{Milnor} and called the Milnor number of $f$ at the origin, provides 
information on the local geometry of $f$ and also information about the local 
topology of the hypersurface $X = f^{-1}(0)$. For example, when $f$ has an 
isolated critical point at the origin, the following invariants coincide up 
to sign:

\begin{itemize}
 \item[(a)] the Milnor number of $f$ at the origin; 
 \item[(b)] the number of Morse critical points of a morsefication of $f$; 
 \item[(c)] the Poincar\'e-Hopf index of the complex conjugate of the gradient vector field of $f$.
\end{itemize}

Let $(X,0)$ be a pure-dimensional germ of an analytic singular space embedded in 
$\mathbb{C}^n$ and $f:(X,0) \to (\mathbb{C},0)$ a germ of analytic function with stratified isolated singularity at the origin. 
Brasselet et al. introduced in \cite{BMPS} a generalization of $(c)$, called 
Euler obstruction of $f$, denoted by ${\rm Eu}_{f,X}(0)$. Roughly, it is the 
obstruction to extending a lifting of the conjugate of the gradient vector 
field of $f$ as a section of the Nash bundle of $(X,0)$. It is then natural
to compare ${\rm Eu}_{f,X}(0)$ to several generalizations of the Milnor number for a function $f$ on a singular germ $(X,0)$, which was nicely treated in \cite{STV}.

Dutertre and Grulha \cite{DG} proved that, for a function germ $f:(X,0)\to 
(\mathbb{C},0)$ with stratified isolated singularity at the origin, the difference 
${\rm Eu}_{X}(0)-{\rm Eu}_{f,X}(0)$ can be computed in terms of the relative polar 
varieties. Here ${\rm Eu}_{X}$ denotes the famous ``local Euler obstruction" constructible function of MacPherson \cite{MacPherson}. Even if $f$ has a stratified non-isolated singularity, Dutertre and
Grulha \cite{DG} introduced the Brasselet number $B_{f,X}(0)$ in terms of a good $a_f$-stratification via
$$B_{f,X}(0) = \chi(X \cap f^{-1}(\delta)\cap B_{\varepsilon}, \rm{Eu}_X) \ \ \ \text{for} \ \ \ 0<|\delta|\ll\varepsilon\ll1,$$
with $B_{\varepsilon}$ a small open ball of radius $\varepsilon$ (in some local embedding).
So this is the corresponding weighted Euler characteristic of the Milnor fiber $X \cap f^{-1}(\delta)\cap B_{\varepsilon}$. In the important special case of
a stratified isolated critical point of $f$, this Brasselet number reduces by \cite[Theorem $3.1$]{BMPS} (as recalled here in Theorem \ref{Euler obstruction of a function formula})
to the difference
$$B_{f,X}(0) = {\rm Eu}_X(0) - {\rm Eu}_{f,X}(0).$$
The Brasselet 
number $B_{f,X}(0)$ provides interesting results, like the L\^e-Greuel type formula 
proved in \cite{DG}. This invariant is also closely related to the Euler defect
$$D_{f,X}(0) = {\rm Eu}_X(0) - B_{f,X}(0),$$
as defined in \cite{BMPS}, which also deals with holomorphic germ-functions with arbitrary singularities.

Topological objects associated to complex functions such as the Milnor number, the Euler obstruction of a function and the Brasselet number play an useful rule in the study of equisingularity, specially in terms of Whitney equisingularity. Let $f(t,\bold{z})=f(t,z_1,\ldots,z_n)$ be a non-constant polynomial function in $\mathbb{C}\times\mathbb{C}^n$, such that $f(t,0)=0,$ for all small $t.$ Denote $f_t(\bold{z})=f(t,\bold{z})$ and let $V(f_t)$ be the hypersurface in $\C^n$ given by the zeros of $f_t.$ The family $\{V(f_t)\}_t$ is Whitney equisingular if there exists a Whitney $(b)-$ regular stratification of $V(f)$ such that the $t$- axis $\C\times\{0\}$ is a stratum. A result about this behavior was given by Briançon in \cite{Briancon} for families of isolated hypersurface singularities. Briançon \cite{Briancon} gives a sufficient condition for a family to be Whitney equisingular when the Newton boundary  of $f_t$ is independent of $t$ and $f_t$ is (Newton) non-degenerate.  The Whitney equisingularity combined with the Thom-Mather first isotopy theorem implies topological equisingularity, that is, the local ambient topological type of $V(f_t)$ at $0$ is independent of $t,$ for a small $t.$ Hence, in this case, the Milnor number of $f_t$ is constant on the given family, for $t$ small enough. 

For non-isolated singularities, as one may expect, Whitney equisingularity is more delicate. In this setting, Eyral and Oka have several results for some classes of singularities. For example, in \cite[Theorem 3.8]{EO}, they prove the Whitney equisingularity for a family of non-degenerate functions with constant Newton boundary and satisfying an uniformly ``locally tame" condition. In a more general setting, in \cite[Theorem $3.3$]{EO2} (recalled here in Theorem \ref{mt}), they prove that if a family of non-constant polynomial functions $\{p_t\}_t$ is Newton-admissible (see Definition \ref{Newton}) then the family of complete intersection varieties $X_t$ associated to $\{p_t\}_t$ is Whitney equisingular. In \cite[Theorem $5.3$]{EO2} they show that the Milnor
fibrations of $p_t$ and $p_0$ at $0$ are isomorphic for small $t$. So it is
natural to ask about the constancy of the Euler obstruction or
Brasselet number of the functions $p_t$ for small $t$. This will be
positively answered by the main results of this paper, based in
addition on some results of Matsui-Takeuchi \cite{MT1} for expressing
suitable invariants in terms of volumes of Newton polyhedra.



We focus on Eyral and Oka's setting  \cite{EO2}. For coordinates $(t, z) := (t, z_1, \dots ,z_n)$ in $\mathbb{C} \times \mathbb{C}^n$, and for any
$k \in K_0 := \{1, \dots , k_0\}$, with a given $k_0\geq 3$, let $f^k
: \mathbb{C} \times \mathbb{C}^n \to \mathbb{C}$ be a non-constant polynomial
function satisfying $f^k(t, 0) = 0$, for all $t$. Define the product function
$p(t, z) := f^1(t, z) \cdots f^{k_0}(t, z)$, $p_t(z) := p(t, z)$ and $f^k_t(z) := f^k(t, z)$. Let $X_t \subset \mathbb{C}^n$ be the germ
given by $V(f^1_t, \dots, f^{k_0 -2}_t)$ and $f_t=f_t^{k_0 -1}, g_t=f_t^{k_0}$ non-constant polynomial functions on $X_t$. Our main results can now be stated as follows. Let $X_t^{g_t} := X_t \cap g_t^{-1}(0)$. If the family $\{p_t\}_t$ is Newton-admissible (see Definition \ref{Newton}), then we get in Proposition \ref{Corolario 4.3 para variedades singulares} (for $0<|\delta|\ll\varepsilon\ll1$):
\begin{equation}\label{eq1nv}
 B_{f_t,X_t}(0)-\chi(X_t^{g_t} \cap f_t^{-1}(\delta)\cap B_{\varepsilon}, {\rm Eu}_{X_t})=(-1)^{d-1}m_t,
\end{equation}
with $d = dim_{\mathbb{C}}X_t$ and $m_t$ the number of stratified Morse critical
points of a morsefication of $g_t : X_t^{g_t}\cap f_t^{-1}(\delta) \cap B_{\varepsilon} \to \mathbb{C}$ appearing
on $(X_t)_{reg} \cap f_t^{-1}(\delta) \cap \{g_t \neq 0\} \cap B_{\varepsilon}$. This can be seen as an
extension of a Lê-Greuel type formula of Dutertre-Grulha \cite[Theorem $4.4$]{DG} (recalled here in Theorem \ref{Le Greuel}) in this Newton-admissible context without any assumption on the dimension
of the stratified critical locus $\Sigma_{\mathcal{V}_{f_t}} g_t$ of the good Whitney stratification $\mathcal{V}_{f_t}$ of $X_t$ relative to $f_t$ given in Remark \ref{remark34}. Note that
in this case $\Sigma_{\mathcal{V}_{f_t}} g_t \subset \{f_t =0\}\cup\{g_t=0\}$ is a union of strata of $\mathcal{V}_{f_t}$ (by the proof of Lemma \ref{dimensionsing}). With the same assumptions
and notations, the main result Theorem \ref{diferencadenumerosdebrasselet1} can be formulated as
\begin{equation}\label{eq2nv}
 B_{f_t,X_t}(0)- B_{f_t,X_t^{g_t}}(0) -\chi(X_t^{g_t} \cap f_t^{-1}(\delta)\cap B_{\varepsilon}, {\rm Eu}_{X_t^{g_t}}-{\rm Eu}_{X_t})=(-1)^{d-1}m_t,
\end{equation}
where $m_t$ is the number of stratified Morse critical points of a partial 
morsefication (Definition \ref{partial morsefication}) of $g_t: X_t\cap f_t^{-1}(\delta) \cap B_{\varepsilon} \to 
\mathbb{C}$ appearing on $({X_t})_{\rm reg} \cap f_t^{-1}(\delta) \cap 
 \{g_t 
\neq 0\} \cap B_{\varepsilon}$. 
Note that the support of the constructible function ${\rm Eu}_{X_t^{g_t}}-{\rm Eu}_{X_t}|X_t^{g_t}$ is a union of strata of $\Sigma_{\mathcal{V}_{f_t}} g_t$ contained in $\{g_t=0\}$.
In case of a generic linear form $f_t$ with respect to $X_t$, Corollary
\ref{corollary39} states the following counterpart:
\begin{equation}\label{eq3nv}
 {\rm Eu}_{X_t}(0)- {\rm Eu}_{X_t^{g_t}}(0) -\chi(X_t^{g_t} \cap f_t^{-1}(\delta)\cap B_{\varepsilon}, {\rm Eu}_{X_t^{g_t}}-{\rm Eu}_{X_t})=(-1)^{d-1}m_t.
\end{equation}

This paper is organized as follows. In Section $2$ we present some background material concerning the Euler obstruction, Brasselet number and Newton admissible family of non-constant polynomials, which 
will be used in the entire work. In Section $3$,  given a Newton-admissible family $\{f_t\}_t$, we consider {the germ at 
the origin} of a locally tame complete intersection variety
 $(X_t,0)$ given by $V(f_t^1, \dots, f_t^{k_0 -2})$ and $f_t=f_t^{k_0 -1}, g_t=f_t^{k_0}$ non-constant polynomial functions on $X_t$. We construct a good stratification 
$\mathcal{V}_{f_t}$ of {the representative $X_t$} re\-la\-ti\-ve to $f_t$, also 
a good stratification $\mathcal{V}_{f_t}^{g_t}$ of $X_t^{g_t}$ re\-la\-ti\-ve to $f_t$ and we prove our main result as stated above. To do this, we first present a generalization of \cite[Corollary 4.3]{DG}. In Section $4$ we establish some
formulae for the Brasselet numbers $B_{f_t,X_t}(0)$ and $B_{f_t,X_t^{g_t}}
(0)$ in
terms of volumes of Newton polyhedra (see formulae (\ref{NB2}) and
(\ref{NB1})), based in addition on some results of Matsui-Takeuchi \cite{MT1}. As an application, we obtain in Corollary \ref{corollary43} that the
number of Morse critical points $m_t$ as in (\ref{eq1nv}) and (\ref{eq2nv}) above
does not depend on $t$ for $t$ small enough, and the same is true
for the Brasselet numbers $B_{f_t,X_t}(0)$ and $B_{f_t,X_t^{g_t}}
(0)$.

\section{Preliminary notions and results}

In this section, we provide the necessary background to develop our results.

\subsection{Euler obstruction}


The local Euler obstruction was defined by MacPherson in 
\cite{MacPherson} as a tool to prove the conjecture about the existence and 
unicity of the Chern classes in the singular case. Since then it has been extensively investigated by many
authors such as Brasselet and Schwartz \cite{BS}, Sebastiani \cite{Sebastiani}, L\^{e} and Teissier \cite{LT}, Sabbah \cite{Sabbah}, 
Dubson \cite{Dubson}, Kashiwara \cite{Kashiwara} and others.

Let $(X,0) \subset (\mathbb{C}^n,0)$ be a pure-dimensional complex analytic subset $X \subset U$ of an open set $U \subset \mathbb{C}^n$. We consider a complex analytic Whitney stratification $\mathcal{V} = \{V_i\}$ of $U$ adapted to $X$ (i.e. $X$ is a union of strata) and we assume that $\{0\}$ is a stratum. We choose a representative $X$ small enough of $(X,0)$ such that $0$ belongs to the closure of all the strata. We write $X= \cup_{i=0}^q V_i$ where $V_0 = \{0\}$ and $V_q = X_{\rm reg}$ is the set of regular points of $X$. We assume that the strata $V_0,\ldots,V_{q-1}$ are connected. Note that the closures $\overline{V_0},\ldots,\overline{V_{q-1}}$ are complex analytic subsets of $U$.

Let $G(d,n)$ be the Grassmannian manifold and $\tilde{X}$ the Nash modification of $X.$ Consider the extension $\mathcal{T}$ of the tautological bundle over $U\times G(d,n).$ Since \linebreak$\tilde{X}\subset U\times G(d,n)$, we consider $\tilde{\mathcal{T}}$ the restriction of $\mathcal{T}$ to $\tilde{X},$ called the \textbf{Nash bundle}\index{Nash bundle}, and $\pi:\tilde{\mathcal{T}}\rightarrow\tilde{X}$ the projection of this bundle.

In this context, denoting by $\varphi$ the natural projection of $U\times G(d,n)$ at $U.$ Con\-si\-de\-ring $\vert\vert z\vert\vert=\sqrt{z_1\overline{z_1}+\cdots+z_n\overline{z_n}}$, the $1$-differential form $w=d\vert\vert z\vert\vert^2$ over $\mathbb{C}^n$ defines a section in $T^{*}\mathbb{C}^n$ and its pullback $\varphi^{*}w$ is a $1$- form over $U\times G(d,n).$ Denote by $\tilde{w}$ the restriction of $\varphi^{*}w$ over $\tilde{X}$, which is a section of the dual bundle $\tilde{\mathcal{T}}^{*}.$

Choose $\epsilon$ small enough for $\tilde{w}$ be a nonzero section over $\nu^{-1}(z), 0<\vert\vert z \vert\vert\leqslant\epsilon,$ let $B_{\epsilon}$ be the closed ball with center at the origin with radius $\epsilon$ and denote by

\begin{enumerate}

\item $Obs(\tilde{\mathcal{T}}^{*},\tilde{w})\in\mathbb{H}^{2d}(\nu^{-1}(B_{\epsilon}),\nu^{-1}(S_{\epsilon}),\mathbb{Z})$ the obstruction for extending $\tilde{w}$ from $\nu^{-1}(S_{\epsilon})$ to $\nu^{-1}(B_{\epsilon});$

\item $O_{\nu^{-1}(B_{\epsilon}),\nu^{-1}(S_{\epsilon})}$ the fundamental class in $\mathbb{H}_{2d}(\nu^{-1}(B_{\epsilon}),\nu^{-1}(S_{\epsilon}),\mathbb{Z}).$ 
\end{enumerate}

\begin{definition}
The \textbf{local Euler obstruction}\index{local Euler obstruction} of $X$ at $0, \ Eu_X(0),$ is given by the evaluation $$Eu_X(0)=\langle Obs(\tilde{\mathcal{T}}^{*},\tilde{w}),O_{\nu^{-1}(B_{\epsilon}),\nu^{-1}(S_{\epsilon})}\rangle.$$
\end{definition}

In \cite[Theorem 3.1]{BLS}, Brasselet, L\^e and Seade proved a formula to 
compute the local Euler obstruction using generic linear forms.

\begin{theorem}\label{BLS}
Let $(X,0)$ and $\mathcal{V}$ be given as before, then for each generic 
linear form $l$, there exists $\varepsilon_0$ such that for any $\varepsilon$ with 
$0<\varepsilon<\varepsilon_0$ and $\delta\neq0$ sufficiently small, the Euler 
obstruction of $(X,0)$ is equal to 
\[
{\rm Eu}_X(0)=\sum^{q}_{i=1}\chi(V_i\cap B_{\varepsilon}\cap l^{-1}(\delta)) \cdot 
{\rm Eu}_{X}(V_i),
\]
where $\chi$ is the Euler characteristic, ${\rm Eu}_{X}(V_i)$ is the Euler 
obstruction of $X$ at a point of $V_i, \ i=1,\ldots,q$ and 
$0<|\delta|\ll\varepsilon\ll1$. In terms of weighted
Euler characteristics this can be reformulated as
$${\rm Eu}_X(0)=\chi(X\cap l^{-1}(\delta)\cap B_{\varepsilon},{\rm Eu}_X) \ \ \ \text{for} \ \ \ 0<|\delta|\ll\varepsilon\ll1.$$
\end{theorem} 


In the following, we use the notion of stratified critical locus of a function introduced by Massey \cite{MR1404920}.
\begin{definition}
The \textbf{critical locus of $f$ relative to $\mathcal{V}$}, 
$\Sigma_{\mathcal{V}}f$, is defined by the union 
\begin{equation*}
	\Sigma_{\mathcal{V}}f=\bigcup_{V_{\lambda}\in\mathcal{V}}\Sigma(f|_{V_{\lambda}}).
\end{equation*}
\end{definition}

In \cite{BMPS}, Brasselet, 
Massey, Parameswaran and Seade give the definition of an invariant associated to a holomorphic function $f:X\rightarrow\mathbb{C}$ with a stratified isolated singularity at the origin. This invariant can be related to the local Euler obstruction through  the following formula, and it is called the local Euler obstruction of a function.

\begin{theorem}[\cite{BMPS}, Theorem 3.1]\label{Euler obstruction of a function formula}
Let $(X,0)$ and $\mathcal{V}$ be given as before and let 
\linebreak$f:(X,0)\rightarrow(\mathbb{C},0)$ be a function with an isolated 
singularity at $0$. For $0<|\delta|\ll\varepsilon\ll1$, we have
 $${\rm Eu}_{f,X}(0)={\rm Eu}_X(0)-\sum_{i=1}^{q}\chi(V_i\cap B_{\varepsilon}\cap f^{-1}(\delta)) \cdot {\rm Eu}_X(V_i).$$
 In terms of weighted
Euler characteristics this can be reformulated as
$${\rm Eu}_{f,X}(0)= {\rm Eu}_{X}(0) - \chi(X\cap f^{-1}(\delta)\cap B_{\varepsilon},{\rm Eu}_X) \ \ \ \text{for} \ \ \ 0<|\delta|\ll\varepsilon\ll1.$$
\end{theorem}



{In the stratified case, as we consider $\{0\}$ a stratum, how could we 
``measure'' the degeneracy of $f$ at this point? In order to have a good 
generalization of a morsefication in the singular case, we need to deal with 
the contribution of the variety at a point in a $0$-dimensional strata. The 
idea is to characterize a kind of ``Morse'' point in this setting. These 
points are the generic points defined below, following \cite[page 971]{MR1404920}}.

\begin{definition}
Let $\mathcal{V}=\{V_{\beta}\}$ be a complex Whitney stratification of a reduced complex 
analytic space $X$ and $p$ be a point in a stratum $V_{\beta}$ of 
$\mathcal{V}$. A \textbf{degenerate tangent plane of $\mathcal{V}$ at $p$} is 
an element $T$ of some Grassmanian manifold such that 
$T=\displaystyle\lim_{p_i\rightarrow p}T_{p_i}V_{\alpha}$, where $p_i\in 
V_{\alpha}$, for some $V_{\alpha}\neq V_{\beta}$.
\end{definition}

\begin{definition}
Let $(X,x)\subset(U,x)$ be a germ of complex analytic space in $\mathbb{C}^n$ 
equipped with a Whitney stratification and let 
$f:(X,x)\rightarrow(\mathbb{C},0)$ be an analytic function, given by the 
restriction of an analytic function $F:(U,x)\rightarrow(\mathbb{C},0)$. Then 
$x$ is said to be a \textbf{generic point}\index{holomorphic function 
germ!generic point of} of $f$ if $Ker(d_xF)$ is transverse in 
$\mathbb{C}^n$ to all degenerate tangent planes of the Whitney stratification 
at $x$, where $U$ is an open set of $\mathbb{C}^n$ and $d_xF$ denotes the derivative of $F$ at $x.$
\end{definition}

The definition of a morsefication of a function is given as follows. 

\begin{definition}
Let $\mathcal{W}=\{W_0,W_1,\ldots,W_q\}$, with $x\in W_0$, be a Whitney 
stratification of the complex analytic space $X$. A function 
$f:(X,x)\rightarrow(\mathbb{C},0)$ has in $x$ a \textbf{stratified Morse critical point}, if $x$ is a generic point of $f|_{W_i}$ for all $i \neq 0$, and $f|_{W_0}: W_0\rightarrow\mathbb{C}$ has in $x$ a Morse critical point in case
$\dim W_0\geq1$.
A \textbf{stratified morsefication}\index{holomorphic function 
germ!stratified morsefication of} of a germ of analytic function 
$f:(X,x)\rightarrow(\mathbb{C},0)$ is a deformation $\tilde{f}$ of $f$ such 
that $\tilde{f}$ has only stratified Morse critical points.
\end{definition}

Using the previous definitions, we can now state  Seade, Tib\u{a}r and Verjovsky result \cite[Proposition 2.3]{STV}.

\begin{proposition}\label{Eu_f and Morse points}
Let $f:(X,0)\rightarrow(\mathbb{C},0)$ be a germ of analytic function with stratified isolated singularity at the origin, with $X$ pure $d$-dimensional. Then, \begin{center}
${\rm Eu}_{f,X}(0)=(-1)^dm$,
\end{center}
where $m$ is the number of Morse points in $X_{\rm reg}$ (in a
small open neighborhood of $0$) in a 
stratified 
morsefication of $f$.
\end{proposition}

\subsection{Stratifications and Brasselet number}

Let $(X,0)$ be an equidimensional complex analytic germ and let $f : (X,0) \to (\mathbb{C}, 0)$ be a holomorphic function-germ. Through this work, we use the following notation: for 
subsets $A\subset \C^n$, $B\subset \C^m$ and a function $f:A\to B$, $A^f:= A\cap f^{-1}(0)$.


\begin{definition}\label{good stratification}
	A \textbf{good stratification} of $X$ relative to $f$ is a stratification 
	$\mathcal{V}$ of $X$ which is adapted
	to $X^{f}$ (\ie $X^{f}$ is a union of strata) satisfying the following conditions: $\bigsetdef{V_i \in 
	\mathcal{V}}{ V_i \not\subset X^{f}}$ is a Whitney 
	stratification
	of $X \setminus X^{f}$ and for any pair of strata 
	$(V_{\alpha},V_{\beta})$ such that $V_{\alpha} \not\subset X^{f}$ and 
	$V_{\beta} \subset X^{f}$, the $(a_f)$-Thom condition is satisfied. We 
	call the strata included in $X^{f}$ the good strata.
\end{definition}

By \cite{GM}, given a stratification $\mathcal{S}$ of $X$, one can refine 
$\mathcal{S}$ to obtain a Whitney stratification $\mathcal{V}$ of $X$ which 
is adapted to $X^{f}$. Moreover, by \cite[Theorem $4.3.2$]{BMM} (see also \cite{P}), the refinement 
$\mathcal{V}$ satisfies the $(a_f)$-Thom condition. This means that good 
stratifications 
always exist. 


For instance, if $\mathcal{V}$ is a Whitney stratification of $X$ and $f : X \to 
\mathbb{C}$ has a stratified isolated
critical point, then the set
\[
\bigsetdef{V_{\alpha} \setminus X^{f}, \ \ V_{\alpha} \cap 
X^{f} \setminus \left\{0 \right\}, \ \ \left\{ 0\right\}}{\ \ V_{\alpha} \in 
\mathcal{V}}
\]
is a good stratification of $X$ relative to $f$. We call it the {\it good 
stratification induced by $f$}.


Durtertre and Grulha \cite{DG} defined the Brasselet number as follows.

\begin{definition}\label{BrasseletN}
	Let $\mathcal{V} = \left\{V_i 
	\right\}_{i=0}^{q}$ be a good stratification of $X$ relative
	to $f$. The \textbf{Brasselet number}, ${\rm B}_{f,X}(0)$, is defined by 
	\[
	{\rm B}_{f,X}(0) = \sum_{i 
	=1}^q\chi \big(V_i\cap B_{\varepsilon}(0) \cap f^{-1}(\delta) \big) \cdot
	{\rm Eu}_{X}(V_i),
	\] where $0< \left| \delta \right| \ll \varepsilon \ll 1$. In terms of weighted
Euler characteristics, this can be reformulated as
$${\rm B}_{f,X}(0)=\chi(X\cap f^{-1}(\delta)\cap B_{\varepsilon},{\rm Eu}_X) \ \ \ \text{for} \ \ \ 0<|\delta|\ll\varepsilon\ll1.$$
\end{definition}

Many results using Brasselet number use some auxiliary definitions which are presented in the sequence. Let  $g :(X, 0)\rightarrow(\mathbb{C},0)$ be a function-germ.

\begin{definition}
If $\mathcal{V}=\{V_{\lambda}\}$ is a stratification of $X$, the 
\textbf{symmetric relative polar variety of $f$ and $g$ with respect to 
$\mathcal{V}$}, $\tilde{\Gamma}_{f,g}(\mathcal{V})$, is the union 
$\cup_{\lambda}\tilde{\Gamma}_{f,g}(V_{\lambda})$, where 
$\tilde{\Gamma}_{f,g}(V_{\lambda})$ denotes the closure in $X$ of the critical locus 
of $(f,g)|_{V_{\lambda}\setminus (X^f\cup X^g)}$. 
\end{definition}

{Using these varieties, we can introduce the notion of tractability (following Massey \cite{MR1404920}).}

{\begin{definition}\label{definition tractable}
A function $g :(X, 0)\rightarrow(\mathbb{C},0)$ is \textbf{tractable at the
origin with respect to a good stratification $\mathcal{V}$ of $X$ relative to 
$f :(X, 0)\rightarrow(\mathbb{C},0)$} if the dimension of 
$\tilde{\Gamma}_{f,g}(\mathcal{V})$ is less or equal to 1 in a 
neighborhood of the origin and, for all strata 
$V_{\alpha}\subseteq X^f$,
$g|_{V_{\alpha}}$ has no critical points in a neighborhood of the origin 
except perhaps at the origin itself.
\end{definition}}

The following result shows that the Brasselet number satisfies a L\^e-Greuel 
type formula \cite[Theorem 4.4]{DG}. 

\begin{theorem}\label{Le Greuel}
	Suppose $X$ is pure $d$-dimensional and that $	\Sigma_{\mathcal{V}}g=\{0\}$. Then
	\[
	{\rm B}_{f,X}(0) - {\rm B}_{f,X^{g}}(0)  = (-1)^{d-1}m, 
	\]
	where $0 < \left|\delta\right| \ll \varepsilon \ll 1$ and $m$ is the number of stratified Morse critical points of a 
	morsefication of $g:X\cap f^{-1}(\delta)\cap 
	B_{\varepsilon}\rightarrow\mathbb{C}$ appearing on $X_{\rm reg}\cap 
	f^{-1}(\delta)\cap \{g\neq 0\}\cap B_{\varepsilon}$. In particular, this number does not
	depend on the morsefication.
\end{theorem}  

  This formula is due to a more general result \cite[Corollary 4.3]{DG}. And in Section $3$ we will use the concept of Newton-admissible family to present a version of this result, without any hypothesis on the
dimension of the singular set of f or g. 

In \cite[Theorem $3.2$]{Santana}, Santana considered the case where the 
function $g$ has a stratified singular set of dimension $1$ and generalized \cite[Theorema 4.4]{DG}. For that, we need the notion of partial morsefications introduced by Dutertre and Grulha in \cite{DG}.

\begin{definition}\label{partial morsefication}
A partial morsefication of $g:f^{-1}(\delta)\cap X\cap B_{\epsilon}\rightarrow\mathbb{C}$ is a function \linebreak$\tilde{g}: f^{-1}(\delta)\cap X\cap B_{\epsilon}\rightarrow\mathbb{C}$ (not necessarily holomorphic) which is a local morsefication of all isolated critical points of $g$ in $f^{-1}(\delta)\cap X\cap \{g\neq 0\}\cap B_{\epsilon}$ and which coincides with $g$ outside a small neighborhood of these critical points.
\end{definition}

With this definition, we enunciate Santana's result \cite[Theorem $3.2$]{Santana}.
\begin{theorem}
Suppose that $g$ is tractable at the origin with respect to  $\mathcal{V}$ 
relative to $f$. Then, for $0<|\delta|\ll\varepsilon\ll1$, 
\[
B_{f,X}(0)-B_{f,X^g}(0)-\sum_{j=1}^{r}m_{f,b_j}\cdot({\rm Eu}_X(b_j)-{\rm Eu}_{X^g}(b_j))=(-1)^{d-1}m,
\]
where $m$ is the number of stratified Morse critical points of a partial 
morsefication of $g:X\cap f^{-1}(\delta)\cap 
B_{\varepsilon}\rightarrow\mathbb{C}$ appearing on $X_{\rm reg}\cap 
f^{-1}(\delta)\cap \{g\neq 0\}\cap B_{\varepsilon}$. Here
$\Sigma_{\mathcal{V}}g=\{0\}\cup b_1\cup\ldots\cup b_r \subset X^g = X \cap g^{-1}(0)$ is a stratification of $\Sigma_{\mathcal{V}}g$ with $b_j$ a one-dimensional stratum contained in some $V_{\alpha}$ (or empty), and $m_{f,b_j}$ is
the multiplicity of $f|_{b_j}$. By taking the germ small enough, one
also has that the constructible functions ${\rm Eu}_X$ and ${\rm Eu}_{X^g}$ are
constant on all $b_j$ .
\end{theorem}

\subsection{Non-degenerate locally tame complete intersection variety}

Let us present the definition of non-degenerate locally tame complete intersection variety and the necessary 
background in order to state our results. In this section, we follow the definitions and notations presented by Eyral and Oka \cite{EO2, Oka3}. 

Denote by $z := (z_1, \dots , z_n)$ the coordinates in $\mathbb{C}^n$, and by $f(z) = \sum_{\alpha} c_{\alpha}z^{\alpha}$ a non-constant polynomial function which vanishes at $0 \in \mathbb{C}^n$. The $n$-tuple
$\alpha := (\alpha_1, \dots, \alpha_n)$ is an integer vector, $c_{\alpha} \in \mathbb{C}$, and $z^{\alpha}$ denotes the monomial $z_1^{\alpha_1} \cdots z_n^{\alpha_n}$. Moreover, for any subset $I \subseteq \{1, \dots , n\}$, we have the following sets
$$ \mathbb{C}^I :=\{ (z_1, \dots , z_n) \in \mathbb{C}^n; \ \ z_i = 0 \ \ \text{if} \ \ i \notin I \},
$$
$$
{\mathbb{C}^{*}}^{I} := \{ (z_1, \dots , z_n) \in \mathbb{C}^n; \ \ z_i = 0 \ \  \text{if and only if}  \ \ i \notin I \}.
$$
When $I = \emptyset$, we have $\mathbb{C}^{\emptyset} = {\mathbb{C}^{*}}^{\emptyset} = \{0\}$. If $I =\{1, \dots,n\}$, we have ${\mathbb{C}^{*}}^{\{1,\dots,n\}} = (\mathbb{C}^{*})^{n}$, where ${\mathbb{C}^{*}}:= \mathbb{C}\setminus \{0\}$. 

 The Newton polyhedron of the germ $f: (\mathbb{C}^n,0) \to (\mathbb{C},0)$ is the convex hull in $\mathbb{R}^n_{+}$ (which denotes the positive orthant of $\mathbb{R}^n$) of the set 
 $$
 \bigcup_{c_{\alpha}\neq 0} (\alpha + \mathbb{R}^{n}_{+}),
 $$
and we denote it by $\Gamma_+(f)$. 

Given a non-zero weight vector $w := (w_1, \dots , w_n) \in \mathbb{ N}^n \setminus \{0\}$, we denote by $l_w$ the restriction to $\Gamma_+(f)$ of the linear map  $L:\mathbb{R}^n \to \mathbb{R}$ defined by
$$
L(x_1,\dots,x_n) = w_1 x_1 + \dots + w_n x_n.
$$

As $\Gamma_+(f) \subset \mathbb{R}^n_{+}$, the map $l_w$ has a minimal value, which we denote by $d(w; f)$. The minimum locus in $\Gamma_+(f)$, defined by
$$
\Delta(w;f) = \{ x \in \Gamma_+(f); \ \ l_w(x) = d(w; f)\},
$$
is a face of $\Gamma_+(f)$. The union of the
compact faces of $\Gamma_+(f)$ is called the Newton boundary  of $f$, and we will represent it by $\Gamma(f)$.

If $w_i > 0$ for all $i$, then $\Delta(w; f)$ is a compact face of $\Gamma(f)$. Moreover, the non-compact Newton boundary is the union of the usual Newton boundary $\Gamma(f)$ together with the \textbf{essential} non-compact faces of $\Gamma_+(f)$, that is, the
non-compact faces $\Delta(w; f)$ for which the restriction $f|_{\mathbb{C}^{I(w)}}$
identically
vanishes, where $I(w) := \{i \in \{1, \dots , n\}; \ \ w_i = 0\}$.

Lastly, denoting by $V_f$ the set of all subsets $I \subseteq \{1, \dots , n\}$ satisfying $f|_{\mathbb{C}^I} \equiv 0$, we say that $\mathbb{C}^I$
is a vanishing (respectively, a
non-vanishing) coordinate subspace for $f$, if $I \in V_f$ (respectively, if $I \notin V_f$ ).

In the following, we present the concept of non-degenerate complete intersection (see for instance \cite{Oka}).

\begin{definition}
Consider $k_0$ non-constant polynomial functions $f^1(z), \dots, f^{k_0}(z)$ which
all vanish at the origin. We say that the germ at $0$ of the variety $V(f^1, \dots, f^{k_0}) := \{z \in \mathbb{C}^n; \ \ f^1
(z) = \dots = f^{k_0}(z) = 0\}$
is a germ of a non-degenerate complete intersection variety if for any positive weight vector $w$, the toric variety
$$
{V^*} {(f^1_w, \dots , f^{k_0}_w ) }:= \{z \in (\mathbb{C}^*)^n; \ \ f^1_w(z) = \dots = f^{k_0}_w(z) = 0\}
$$
is a reduced, non-singular, complete intersection variety in $(\mathbb{C}^*)^n$. Here, $f_w^{i}$ denotes the face function of $f^i$ with respect to the weight vector
$w$, i.e. $f_w^{i} = f^i|_{\Delta(w; f^i)}$.
\end{definition}

Here we assume $\{i_1, \dots, i_m\} \subset \{1, \dots, n\}$ with $|\{i_1, \dots, i_m\}| = m$.
Notice that the variety ${V^*} {(f^1_w, \dots , f^{k_0}_w ) }
$ is globally defined in $(\mathbb{C}^*)^n$. 

\begin{remark} The class of non-degenerate singularity is open and dense when the {Newton
boundary} is fixed \cite{Oka}.
\end{remark}

Before presenting the important concept of locally tame complete intersection variety \cite{EO2, Oka3}, let us introduce a necessary notation. For any $u_{i_1}, \dots , u_{i_m} \in \mathbb{C}^*$, with $m \leq n$, let $(\mathbb{C}^*)^n(u_{i_1}, \dots , u_{i_m})$ denote the set of points $(z_1, \dots , z_n) \in (\mathbb{C}^*)^n$ satisfying $z_{i_j} = u_{i_j}$
for $1 \leq j \leq m$. Here we assume $\{i_1, \dots,i_m\} \subset \{1,\dots, n\}$ with $|\{i_1, \dots, i_m\}| = m$.

\begin{definition} The germ at $0$ of $V(f^1, \dots, f^{k_0})$  is called a germ of a
locally tame complete intersection variety if there is a number $R(f^1, \dots , f^{k_0}) >0$ such that for any non-empty subset $I := \{i_1, \dots , i_m\} \in V_{f^1} \cap \dots \cap V_{f^{K_0}}$, with $|\{i_1, \dots, i_m\}| = m$, any non-zero weight vector $w$ with $I(w) = I$, and any
non-zero complex numbers $u_{i_1}, \dots , u_{i_m}$ satisfying the inequality
$$
\sum_{j=1}^{m} | u_{i_j}|^2 < R(f^1, \dots , f^{k_0}),
$$
the toric variety
$$
{V^*} {(f^1_w, \dots , f^{k_0}_w ) } \cap (\mathbb{C}^*)^n (u_{i_1},\dots,u_{i_m})
$$
is a reduced, non-singular, complete intersection variety in $(\mathbb{C}^*)^n (u_{i_1},\dots,u_{i_m})$.
\end{definition}

A number $R(f^1, \dots , f^{k_0}) > 0$
satisfying the above definition is called a radius of local tameness of the functions $f^1, \dots , f^{k_0}$.

In \cite{EO, EO2}, Eyral and Oka used the objects described above to study the Whitney equisingularity of families of complete intersection varieties, not necessarily with isolated singularity. In the sequence, we are going to introduce some more notations and definitions in order to state Eyral and Oka's result concerning Whitney equisingularity. We use Eyral and Oka’s notation.

Let $(t, z) := (t, z_1, \dots ,z_n)$ be coordinates in $\mathbb{C} \times \mathbb{C}^n$, and for any
$k \in K_0 := \{1, \dots , k_0\}$, let $f^k
: \mathbb{C} \times \mathbb{C}^n \to \mathbb{C}$ be a non-constant polynomial
function satisfying $f^k(t, 0) = 0$, for all $t$. Define the product
$p(t, z) := f^1(t, z) \cdots f^{k_0}(t, z)$ and denote $p_t(z) := p(t, z)$ and $f^k_t(z) := f^k(t, z)$.

\begin{definition}\label{Newton} The family $\{p_t\}_t$ is called Newton-admissible if for any sufficiently small $t$, the following two conditions are satisfied:
\begin{itemize}
\item for any $k \in K_0$, the Newton boundary $\Gamma(f^k_t)$ does not depend on $t$;

\item for any $\{k_1, \dots , k_p\} \subseteq K_0$, the germ at $0$ of $V(f^{k_1}_t, \dots ,f^{k_p}_t)$ is a
germ of a non-degenerate, locally tame, complete intersection variety, and there exists a radius of local tameness $R(f^{k_1}_t, \dots , f^{k_p}_t)$ for the
corresponding functions $f^{k_1}_t, \dots , f^{k_p}_t$ which is greater than
some number $R>0$ independent of $t$ and of the choice of the
subset $\{k_1, \dots , k_p\}$.
\end{itemize}
\end{definition}

In particular, if $\{p_t\}_t$ is Newton-admissible, by \cite[ Lemma (2.8.2)]{Oka}, there is a neighborhood of the origin such that any subset given by 
\begin{align}\label{p3-thesubset}
\bigcap_{k\in K} V^{*I}(f^k)
\end{align}
is non-singular, in which 
\begin{align*}
V^{*I}(f^k):=V(f^k)\cap (\mathbb{C}\times \mathbb{C}^{*I}), 
\end{align*} $K\subseteq K_0$ and $I\subseteq \{1,\ldots,n\}$. Moreover, if for all $k\in K$, $f^k\vert_{\mathbb{C}\times\mathbb{C}^I} \not\equiv 0$, then the subset \eqref{p3-thesubset} is also a complete intersection variety. It follows that the collection $\mathcal{S}$ of all non-empty subsets of the form
\begin{align*}
S^I(K) & := \{(t,\mathbf{z})\in\mathbb{C}\times \mathbb{C}^{*I} \mid 
f^k(t,\mathbf{z})=0 \Leftrightarrow k\in K\}\\
& \ = \bigcap_{k\in K} V^{*I}(f^k) \bigg\backslash \bigcup_{k\in K_0\setminus K} V^{*I}(f^k)
\end{align*}
is a complex analytic stratification of $V(p)$. We call $\mathcal{S}$ the \emph{canonical toric stratification} of $V(p)$. Note that it includes $S^{\emptyset}(K_0)=\mathbb{C}\times\{\mathbf{0}\}$ (i.e., the $t$-axis) as a stratum.

In \cite{EO2}, Eyral and Oka proved the following.

\begin{theorem}(Eyral and Oka)\label{mt}
If the family $\{p_t\}_t$ is Newton-admissible, then the canonical toric stratification $\mathcal{S}$ of $V(p)$ is Whitney $(b)$-regular. In particular, the corresponding family of hypersurfaces $\{V(p_t)\}_t$ is Whitney equisingular. 
\end{theorem}

The special case of hypersurfaces can be found in \cite{EO}.

\begin{remark}\label{remarkstratification}If the family $\{p_t\}_t$ is Newton-admissible, from Theorem \ref{mt} we can conclude that the stratification $\{S^I(K_0)\}_{I\subseteq \{1,\ldots,n\}}$ of $V(f^1,\ldots,f^{k_0})$ is a Whitney $(b)$-regular stratification  with the $t$-axis as a stratum (see \cite[Remark $3.5$]{EO2}). Moreover, combining  \cite[Corollary $3.4$]{EO2} and \cite[Remark $3.5$]{EO2} the collection of subsets $\{S^I(K_0)\cap (\{t\}\times\mathbb{C}^n)\}_{I\subseteq \{1,\ldots,n\}}$ is a Whitney $(b)$-regular stratification of the complete intersection variety $V(f_t^1,\ldots,f_t^{k_0})$ in a neighborhood of the origin of $\mathbb{C}^n$, whose topology is independent of $t$, for sufficiently small $t$.
\end{remark}

\section{Induced good stratifications for non-degenerate locally tame complete intersections}

Let $(X,0)$ be an equidimensional complex analytic germ and let $f : (X,0) \to (\mathbb{C}, 0)$ be a holomorphic function-germ. In order to get a Santana’s type result, but without any hypothesis on the dimension of the singular set of $f$ or $g$, we have two tasks to do.  We start finding suitable good stratifications of the representatives of $X$ and of $X^g$ near the origin. In the sequence we prove a version of \cite[Lemma 4.1]{MR1404920} for not necessarily isolated singularities (Lemma \ref{Masseyversion}).


Now, using the notations of Section $2.3$, let $(t, z) := (t, z_1, \dots ,z_n)$ be coordinates in $\mathbb{C} \times \mathbb{C}^n$, and for any
$k \in K_0 := \{1, \dots , k_0\}$, let $f^k
: \mathbb{C} \times \mathbb{C}^n \to \mathbb{C}$ be a non-constant polynomial
function satisfying $f^k(t, 0) = 0$, for all $t$. Define the product 
$p(t, z) := f^1(t, z) \cdots f^{k_0}(t, z)$ and denote $p_t(z) := p(t, z)$ and $f^k_t(z) := f^k(t, z)$.

{From now on, let us denote by $X$ a sufficiently small representative of the germ of variety $(X,0)$, in which 
$$
  X^{f} := V(f^1,\ldots,f^{k_0 -1}) \subset X = V(f^1,\ldots,f^{k_0 -2}),
$$
where $f:=f^{k_0 -1}$ is a representative of the function-germ $f^{k_0 -1}:(X,0) \to (\mathbb{C},0)$}, $g:=f^{k_0}$ is a representative of the function-germ $f^{k_0}:(X,0) \to (\mathbb{C},0)$ and $k_0 \geq 3$.

\begin{example}
Let us consider $f^1: \mathbb{C} \times \mathbb{C}^4  \to \mathbb{C}$ given by 
    $$
     f^1(t, z_1,z_2,z_3,z_4) = z_1^2z_3^2 - z_2^3z_3^2 + z_3^2 z_4 + z_3^3 + tz_3^5,
    $$
then for each $t$ the function $f^1_t:  \mathbb{C}^4  \to \mathbb{C}$ is given by 
    $$
     f^1_t(z_1,z_2,z_3,z_4) = z_1^2z_3^2 - z_2^3z_3^2 + z_3^2 z_4 + z_3^3 + tz_3^5.
    $$

Since $f_t^1$ is non-degenerate for all values of $t$, by  \cite[ Lemma (2.8.2)]{Oka}
\begin{align}\label{example1}
\bigcap V(f^1_t)\cap (\{t\}\times \mathbb{C}^{*I})
\end{align}
is non-singular for all $t$, where $I\subseteq \{1,2,3,4\}$. Hence the collection $S^I(K_0^2)$ obtained with the above sets gives a stratification of $X_t := V(f^1_t)$, for all $t$, where $K_0^2 = \{1\}$.

For all $t$, the critical set of $f^1_t$ is the subspace                        $$\Sigma_{S^I(K_0^2)}                {f^1_t} = 
\{(z_1,z_2,0,z_4), \ \ z_1,z_2,           z_4 \in\mathbb{C}\}.$$
  
Now defining $f : = f^2: \mathbb{C} \times \mathbb{C}^4  \to \mathbb{C}$ by 
    $$
     f(t, z_1,z_2,z_3,z_4) = z_2^2 - z_3^3 -z_3^2 z_1^2 + 7z_3^2 z_4 + tz_3^7,
    $$
then for each $t$ the function $f_t:  \mathbb{C}^4  \to \mathbb{C}$ is given by 
    $$
     f_t(z_1,z_2,z_3,z_4) = z_2^2 - z_3^3 -z_3^2 z_1^2 + 7z_3^2 z_4 + tz_3^7.
    $$

The critical set of $f_t$ is the subspace $\Sigma {f_t} = \{(z_1,0,0,z_4), \ \ z_1, z_4 \in\mathbb{C}\}$ for all $t$ and we can see that $\Sigma {f_t} \subset \Sigma_{S^I(K_0^2)} {f^1_t} \subset X_t $. Moreover, taking $I_{f_t}=\{1,4\} \subset \{1,2,3,4\}$, then $X_t \cap {\mathbb{C}^{*}}^ {I_{f_t}} =  {\mathbb{C}^{*}}^ {I_{f_t}} = \mathbb{C}^* \times 0 \times 0 \times \mathbb{C}^* \subset \Sigma {f_t}$. Therefore, ${f_t}|_{X_t}$ has a singular set with dimension $2$ at least. Also
\begin{align}\label{example2}
\Big(\bigcap_{k\in K_0^1} V(f^k_t) \Big)\cap (\{t\}\times \mathbb{C}^{*I})
\end{align}
is non-singular, where $K_0^1 = \{1,2\}$ and $I\subseteq \{1,2,3,4\}$. Hence the collection $S^I(K_0^1)$ obtained with the above sets gives a stratification of $X_t^{f_t} := V(f^1_t, f_t)$, for all $t$.

    Let  $g : = f^3: \mathbb{C} \times \mathbb{C}^4  \to \mathbb{C}$ be the function given by $$g(t,z_1,z_2,z_3,z_4)=z_2^2-z_3^2z_1+z_4^3+tz_3^9.$$ For each $t$, the function $g_t:  \mathbb{C}^4  \to \mathbb{C}$ is given by 
    $$
     g_t(z_1,z_2,z_3,z_4) = z_2^2-z_3^2z_1+z_4^3+tz_3^9.
    $$
 Since $\Sigma_{S^I(K_0^2)} f_t^1=\{(z_1,z_2,0,z_4), z_1,z_2,z_4\in\mathbb{C}\}$, we have  $g_t|_{\Sigma_{S^I(K_0^2)} f_t^1}=z_2^2+z_4^3$ and $$\Sigma g_t|_{\Sigma_{S^I(K_0^2)} f^1_t}=\{(z_1,0,0,0),z_1\in\mathbb{C}\}\subset X_t.$$
 
 Taking $I_{g_t}=\{1\},$ we have $X_t \cap {\mathbb{C}^{*}}^ {I_{g_t}} =  {\mathbb{C}^{*}}^ {I_{g_t}} = \mathbb{C}^* \times 0 \times 0 \times 0 \subset \Sigma_{S^I(K_0^2)} {f_t^1}$. 
 Hence, $g_t|_{X_t}$ has singular set of dimension at least $1$.

    Moreover, since $\Sigma f_t\subset \Sigma_{S^I(K_0^2)} f^1_t$, it is sufficient to analyse $g_t|_{\Sigma f_t}$ in order to estimate the dimension of $\Sigma_{S^I(K_0^1)} g_t|_{X_t^{f_t}}.$ We have $g_t|_{\Sigma f_t}=z_4^3$ and $$\Sigma g_t|_{\Sigma f_t}=\{(z_1,0,0,0), z_1\in\mathbb{C}\}.$$ Then $\mathbb{C}^{*{I_{g_t}}} \subset\Sigma_{S^I(K_0^1)} g_t|_{X_t^{f_t}}$.
    
\end{example}


Before we continue, let us denote by $K_0^2:={K}_0 \setminus \{ k_0-1, k_0\}= \{1, \dots , k_0 -2\}$, and by 
$p^2(t, z) := f^1(t, z) \cdots f^{k_0 -2}(t, z)$. Then, using the concept of Newton-admissible family we have the following.





\begin{lemma}\label{any dimensional stratification lemma}
If the family $\{p_t^2\}_t$ is Newton-admissible, then the collection of sets 
\begin{equation*}
\mathcal{V}_f=\left\{S^I(K_0^2) \setminus\{f=0\}, \ \ 
S^I(K_0^2) \cap\{f=0\} \right\} _{I\subseteq \{1,\ldots,n\}}
\end{equation*} 
is a good stratification of $X$ relative to $f$.
\end{lemma}

\begin{proof} 
By Theorem \ref{mt} (see also Remark \ref{remarkstratification}), the set  $\{S^I(K_0^2)\}_{I\subseteq \{1,\ldots,n\}}$ is a Whitney stratification of $X$, therefore  by Theorem p. $99$ of \cite{P} or \cite[Theorem $4.3.2$]{BMM} (see also \cite{Massey1}), we have that $\mathcal{V}_f$ is a good stratification of $X$ relative to $f$.
\end{proof}





Now we aim to construct the appropriate stratification for $X^g,$ in which
$$
  (X^{g})^f := V(f^1,\ldots,f^{k_0 -2}, f^{k_0 -1}, f^{k_0}) \subset X^{g} := V(f^1,\ldots,f^{k_0 -2}, f^{k_0}) \subset X.
$$

As above, let us denote by $K_0^1:={K}_0 \setminus \{k_0 -1\}= \{1, \dots , k_0 -2, k_0\}$, and by 
$p^1(t, z) := f^1(t, z) \cdots f^{k_0 -2}(t,z) \cdot f^{k_0}(t, z)$. Then, using the concept of Newton-admissible family we have the following.





\begin{lemma}\label{any dimensional stratification lemma3}
If the family $\{p_t^1\}_t$ is Newton-admissible, then the collection of sets 
\begin{equation*}
\mathcal{V}^g_f=\left\{S^I(K_0^1) \setminus\{f=0\}, \ \ 
S^I(K_0^1) \cap\{f=0\} \right\} _{I\subseteq \{1,\ldots,n\}}
\end{equation*} 
is a good stratification of $X^g$ relative to $f.$
\end{lemma}

\begin{proof} 
The proof is exactly analogous to the proof of Lemma \ref{any dimensional stratification lemma}.
\end{proof}
 
\begin{remark}\label{remark34} Similarly, considering
$$
  X^{f}_t := V(f^1_t,\ldots,f^{k_0 -1}_t) \subset X_t = V(f^1_t,\ldots,f^{k_0 -2}_t),
$$
$$
  (X^{g})^f_t := V(f^1_t,\ldots,f^{k_0 -2}_t, f^{k_0 -1}_t, f^{k_0}_t) \subset X^{g}_t := V(f^1_t,\ldots,f^{k_0 -2}_t, f^{k_0}_t) \subset X_t,
$$
and applying \cite[Corollary $3.4$]{EO2} we can prove that if the family $\{p_t^2\}_t$ is Newton-admissible, then the collection of sets 
\begin{equation*}\label{element}
\mathcal{V}_{f_t}=\left\{S^I(K_0^2) \cap (\{t\}\times\mathbb{C}^n)\setminus\{f_t=0\}, \ \ 
S^I(K_0^2) \cap (\{t\}\times\mathbb{C}^n) \cap\{f_t=0\} \right\} _{I\subseteq \{1,\ldots,n\}}
\end{equation*} 
is a good stratification of $X_t$ relative to $f_t$, for any sufficiently small $t$. And if the family $\{p_t^1\}_t$ is Newton-admissible, then the collection of sets 
\begin{equation*}
\mathcal{V}^{g_t}_{f_t}=\left\{S^I(K_0^1) \cap (\{t\}\times\mathbb{C}^n)\setminus\{f_t=0\}, \ \ 
S^I(K_0^1) \cap (\{t\}\times\mathbb{C}^n) \cap\{f_t=0\} \right\} _{I\subseteq \{1,\ldots,n\}}
\end{equation*} 
is a good stratification of $X_t^{g_t}$ relative to $f_t$, for any sufficiently small $t$.
\end{remark}

Before proving our main result, we present the following lemma.

\begin{lemma}\label{dimensionsing}
Let $X_t \subset \mathbb{C}^n$ be the germ
given by $V(f_t^1, \dots, f_t^{k_0 -2})$ and $f_t=f_t^{k_0 -1}, g_t=f_t^{k_0}$ non-constant polynomial functions on $X_t$. If the family $\{p_t\}_t$  is Newton-admissible, then the symmetric relative polar variety $\tilde{\Gamma}_{f_t,g_t}(\mathcal{V}_{f_t })$ has dimension less or equal to one, for all $t$ small enough. 
\end{lemma}

\begin{proof} Since the family $\{p_t\}_t$ is Newton-admissible, by \cite[ Lemma (2.8.2)]{Oka}, there is a neighborhood of the origin such that any subset given by 
\begin{align}\label{p3-thesubset2}
V^K_t := \bigcap_{k\in K} V(f^k_t)\cap (\{t\} \times \mathbb{C}^{*I})
\end{align}
is non-singular, where $K\subseteq K_0$ and $I\subseteq \{1,\ldots,n\}$. Moreover, if for all $k\in K$, $f^k_t \vert_{\{t\}\times\mathbb{C}^I} \not\equiv 0$, the subset \eqref{p3-thesubset2} is also a complete intersection variety. Then, if all $k\in K$, $f^k_t\vert_{\{t\}\times\mathbb{C}^I} \not\equiv 0$, the $k$-form 
\begin{equation*}
df^1_t \wedge \cdots\wedge df^{k}_t
\end{equation*}
is nowhere vanishing in  $V^K_t $.

Now, for all $I \subset \{1, \dots, n\}$ such that $f_t\vert_{\{t\}\times\mathbb{C}^I} \equiv 0$, then $\{t\}\times\mathbb{C}^I \subset X^f_t$ (analogously, if $g_t\vert_{\{t\}\times\mathbb{C}^I} \equiv 0$, then $\{t\}\times\mathbb{C}^I \subset X^g_t$). Therefore, the critical locus 
of 
$$(f_t,g_t)|_{V^{I}_t\setminus (X_t^{f_t}\cup X_t^{g_t})}
$$
is empty, in which ${V^{I}_t}$ is a stratum of $\mathcal{V}_{f_t}$. 
\end{proof}




In \cite{DG} Dutertre and Grulha Jr. applied \cite[Theorem 4.2 (A)]{MR1404920} to prove \cite[Corollary 4.3]{DG}.   On the other hand, the essential step in Massey's proof of \cite[Theorem 4.2 (A)]{MR1404920} is \cite[Lemma 4.1]{MR1404920}, which also holds true in our setting. 

\begin{lemma}\label{Masseyversion}
Let $X_t \subset \mathbb{C}^n$ be the germ
given by $V(f_t^1, \dots, f_t^{k_0 -2})$ and $f_t=f_t^{k_0 -1}, g_t=f_t^{k_0}$ non-constant polynomial functions on $X_t$. If the family $\{p_t\}_t$  is Newton-admissible, then, for $\epsilon$ small and nonzero, and $0<\eta\ll\epsilon,$ we may use neighborhoods of the form $B_{\epsilon}\cap g_t^{-1}(\mathbb{D}_{\eta}^{\circ})$ to define the Milnor fibre of $f_t$, where $\mathbb{D}_{\eta}^{\circ}$ denotes the interior of a closed ball centered at the origin and with radius $0<\eta<<1.$
\end{lemma}
\noindent\textbf{Proof.} Following \cite[Lemma 4.1]{MR1404920}, we must show that for $0<\nu,\eta<<a<\epsilon$, the map $$\psi:=(|z|^2,g_t,f_t):X_t\cap \psi^{-1}((a,\epsilon)\times\mathbb{D}^{\circ}_{\eta}\times \mathbb{D}^{\circ}_{\nu})\rightarrow (a,\epsilon)\times\mathbb{D}^{\circ}_{\eta}\times \mathbb{D}^{\circ}_{\nu}$$
 is a proper, stratified submersion.
 
 Then,  using \cite[Proposition 5.2]{EO2}, there exists $\epsilon >0$ such that for all $\epsilon^{\prime}$, $0<\epsilon^{\prime}\leqslant\epsilon,\partial B_{\epsilon^{\prime}}$ intersects transversely the strata $\{g_t^{-1}(0)\cap W_{\beta}, W_{\beta}\subseteq V(f_t)\}$ for all $W_{\beta}\in\mathcal{V}_{f_t}$. 
 
Note that in our context
the $g_t^{-1}(0)\cap W_{\beta}$ are indeed strata of the corresponding stratification. By Lemma \ref{dimensionsing}, the symmetric relative polar variety is at most
one-dimensional. Hence we may choose $\epsilon$ so small that $B_{\epsilon}\cap\tilde{\Gamma}_{f_t,g_t}\cap V(g_t)\subseteq\{0\}$, and so, for any $a$ such that $0<a<\epsilon,$ for all $\eta$ sufficiently small, $$B_{\epsilon}\cap\Gamma_{f_t,g_t}\cap g_t^{-1}(\mathbb{D}^{\circ}_{\eta})\subseteq B_a,$$

 or yet,

 $$\psi^{-1}((a,\epsilon)\times \mathbb{D}^{\circ}_{\eta}\times\mathbb{C})\cap\Gamma_{f_t,g_t}=\emptyset.$$

Now, let $\tilde{g_t}$ be an extension of $g_t$ to an open neighborhood of the origin in $\mathbb{C}^N$, and suppose no matter how small we pick $\eta$ and $\nu$, the map $\psi$ still has critical points. Then there would exist a stratum $W_{\alpha}\nsubseteq V(f_t)$ and a sequence of points $p_i$ in $W_{\alpha}\setminus (V(f_t)\cup V(g_t))$ such that $p_i$ converges to some point $p$ in $V(f_t)\cap V(g_t)-\partial B_a$ and such that 
\begin{eqnarray}\label{sequence1}
    T_{p_i}V(f_t|_{W_{\alpha}}-f(p_i))\cap T_{p_i}V(\tilde{g_t}-\tilde{g_t}(p_i))\subseteq T_{p_i}\partial B_{|p|}.
\end{eqnarray}

Let $W_{\beta}\subseteq V(f_t)$ be the stratum containing $p.$ We may assume that $T_{p_i}V(\tilde{g_t}-\tilde{g_t}(p_i))$ converges to some $\mathcal{I}$ and that  $T_{p_i}V(f_t|_{W_{\alpha}}-f_t(p_i))$ converges to some $\mathcal{T}$. By the good condition, $T_p W_{\beta}\subset \mathcal{T}.$

Since $p_i$ converges to $p$, from \ref{sequence1}, we obtain $\mathcal{T}\cap \mathcal{I}\subset T_p(\partial B_{|p|}).$ Now, since $g_t^{-1}(0)\cap W_{\beta}$ is transversal to $\partial B_{|p|},$ we have \begin{eqnarray}\label{limit1}
    T_p(g_t^{-1}(0)\cap W_{\beta})+ T_p\partial B_{|p|}=\mathbb{C}^N.
\end{eqnarray}

Keeping $f_t=f^{k_0-1}$ and $g_t=f^{k_0}$, let $\tilde{K_0^1}=K_0\setminus\{k_0\}$. Similar to Lemma \ref{any dimensional stratification lemma3}, the collection $$\{S^I(\tilde{K_0^1})\setminus\{g_t=0\}, S^I(\tilde{K_0^1})\cap\{g_t=0\}\}_{I\subseteq \{1,\ldots, n\}}$$ 
is a good stratification of $X_t^{f_t}$ relative to $g_t.$

Moreover, a stratum $S^I(\tilde{K_0^1})\cap\{g_t=0\}$ is precisely of type $S^I(\tilde{K_0^2})\cap\{f_t=0\}$ of stratification $\mathcal{V}_{f_t}$ from Lemma \ref{any dimensional stratification lemma}. 

Notice that $g_t^{-1}(0)\cap W_{\beta}$ is also a stratum of type $S^I(\tilde{K_0^1})\cap\{g_t=0\}$, hence by the good condition $T_p(g_t^{-1}(0)\cap W_{\beta})\subset \mathcal{I}.$

Since $g_t^{-1}(0)\cap W_{\beta}\subset W_{\beta}$, $T_p(g_t^{-1}(0)\cap W_{\beta})\subset T_p(W_{\beta}).$ From $T_p(g_t^{-1}(0)\cap W_{\beta})\subset \mathcal{I},$ it follows that $$T_p(g_t^{-1}(0)\cap W_{\beta})\subset \mathcal{I}\cap T_p(W_{\beta})\subset \mathcal{I}\cap\mathcal{T}\subset T_p(\partial B_{|p|}).$$

This contradicts the transversality obtained in \ref{limit1}.  \qed

Hence, we have the following (just as in \cite[Corollary $4.3$]{DG}):

 \begin{proposition}\label{Corolario 4.3 para variedades singulares}
	Let $X_t \subset \mathbb{C}^n$ be the pure $d$-dimensional germ
given by $V(f_t^1, \dots, f_t^{k_0 -2})$ and $f_t=f_t^{k_0 -1}, g_t=f_t^{k_0}$ non-constant polynomial functions on $X_t$. If the family $\{p_t\}_t$  is Newton-admissible then, for $0 < \left| \delta \right| 
	\ll \varepsilon \ll 1$, 
	\[
	 B_{f_t, X_{t}} (0) -  \sum_{V_i \subset{\mathcal {V}_{f_t}}}  \chi 
\big(V_i \cap X_{t}^{g_t} \cap f_t^{-1}(\delta) \cap B_{\varepsilon} \big) \cdot 
{\rm{Eu_{{X_{t}}}}}(V_i )    =   (-1)^{d-1}m_t,
	\]
 
	\noindent where $0 < \left|\delta\right| \ll \varepsilon \ll 1$ and $m_t$ is the number of stratified Morse critical points of a
	morsefication of $g_t:X_t\cap f_t^{-1}(\delta)\cap 
	B_{\varepsilon}\rightarrow\mathbb{C}$ appearing on $({X_t})_{\rm reg}\cap 
	f_t^{-1}(\delta)\cap \{g_t\neq 0\}\cap B_{\varepsilon}$. In particular, this number does not
	depend on the morsefication. In terms of weighted Euler characteristics
this can be reformulated (for $0<|\delta|\ll\varepsilon\ll1$) as
$$B_{f_t,X_t}(0) - \chi(X_t^{g_t} \cap f_t^{-1}(\delta)\cap B_{\varepsilon},{\rm Eu}_{X_t}) = (-1)^{d-1}m_t.$$

\end{proposition}

Last proposition leads us to our main result in this section. Notice that the strata of $\mathcal{V}^{g_t}_{f_t}$ are the strata of $\mathcal{V}_{f_t}$ intersected by $\{g_t=0\}.$ Let us denote by $V_i^{g_t}$ the strata of $\mathcal{V}_{f_t}$ which intersect the critical set $\Sigma_{\mathcal {V}_{f_t}}g_{t}, i\in\{1,\ldots, q_t\}.$

\begin{theorem}\label{diferencadenumerosdebrasselet1} 	
	Let $X_t \subset \mathbb{C}^n$ be the germ
given by $V(f_t^1, \dots, f_t^{k_0 -2})$ and $f_t=f_t^{k_0 -1}, g_t=f_t^{k_0}$ non-constant polynomial functions on $X_t$. If the family $\{p_t\}_t$  is Newton-admissible then, for $0 < \left| \delta \right| 
	\ll \varepsilon \ll 1$,
\begin{equation}\label{morsepoints}   B_{f_t,X_t}(0)-B_{f_t,X_t^{g_t}}(0)- \displaystyle 
 \sum_{i=1}^{q_t}\chi \big(V_i^{g_t} \cap f^{-1}_t(\delta) \cap B_{\varepsilon} 
\big)   \big(\rm{Eu_{X_t^{g_t}}}(V_i^{g_t}) - 
 \rm{Eu_{X_t}}(V_i^{g_t})  \big)=(-1)^{d-1}m_t
\end{equation}
where $m_t$ is the number of stratified Morse critical points of a partial 
morsefication of $g_t: X_t\cap f_t^{-1}(\delta) \cap B_{\varepsilon} \to 
\mathbb{C}$ appearing on $({X_t})_{\rm reg} \cap f_t^{-1}(\delta) \cap 
 \{g_t 
\neq 0\} \cap B_{\varepsilon}$. In terms of weighted Euler characteristics
this can be reformulated (for $0<|\delta|\ll\varepsilon\ll1$) as
\begin{equation*}
 B_{f_t,X_t}(0)- B_{f_t,X_t^{g_t}}(0) -\chi(X_t^{g_t} \cap f_t^{-1}(\delta)\cap B_{\varepsilon}, {\rm Eu}_{X_t^{g_t}}-{\rm Eu}_{X_t})=(-1)^{d-1}m_t.
\end{equation*}
Note that the support of the constructible function ${\rm Eu}_{X_t^{g_t}} - {\rm Eu}_{X_t}|X_t^{g_t}$ is a union of strata of $\Sigma_{\mathcal{V}_{f_t}} g_t$ contained in $\{g_t =0\}$,
given by our $V_i^{g_t}$ for $i = 1,\dots,q_t$.
\end{theorem}
\begin{proof}
Since the family $\{p_t\}_t$ 
is Newton-admissible and 
$$p(t, z) := f_t^1(t, z) \cdots f_t^{k_0}(t, z),$$
by Definition \ref{Newton} the families $\{p_t^2\}_t$ and $\{p_t^1\}_t$ are also Newton-admissible. Therefore, by Remark \ref{element} the collection  $\mathcal{V}_{f_t}$ is a good stratification of $X_t$, and the collection $\mathcal{V}^{g_t}_{f_t}$ is a good stratification of $X^{g_t}$. Moreover, by Lemma \ref{dimensionsing} and  \cite[Proposition $5.2$]{EO2}, we may apply Lemma \ref{Masseyversion}.

Now, applying Proposition \ref{Corolario 4.3 para variedades singulares}, we have
\begin{equation*}
 B_{f_t, X_t} (0) -  \sum_{V_i \subset{\mathcal {V}_{f_t}}}  \chi 
\big(V_i \cap X_t^{g_t} \cap f_t^{-1}(\delta) \cap B_{\varepsilon} \big) \cdot 
{\rm{Eu_{{X_t}}}}(V_i )    =   (-1)^{d-1}m_t,
\end{equation*}
in which $ 0 < \left| \delta \right| \ll \varepsilon 
\ll 1$.

If $V_i \not\subseteq \Sigma_{\mathcal {V}_{f_t}}g_t$, $V_i$ intersects 
$g_t^{-1}(0)$ transversely, hence ${\rm Eu}_{X_t}(V_i) = {\rm Eu}_{X_t^{g}}(V_i 
\cap g_t^{-1}(0))$. Then
\begin{equation*}
\begin{aligned}
\sum_{V_i \subset{\mathcal {V}_{f_t}}}  \chi 
\big(V_i \cap X_t^{g_t} \cap f_t^{-1}(\delta) \cap B_{\varepsilon} \big) \cdot 
{\rm{Eu_{{X_t}}}}(V_i ) &  =  \sum_{V_i \not\subseteq \Sigma_{\mathcal {V}_{f_t}}g_t}  \chi \big(V_i^{g_t} 
 \cap f_t^{-1}(\delta) \cap B_{\varepsilon} \big)\cdot 
 {\rm{Eu_{{X_t^g}}}}(V_i^{g_t}) \\ 
& + \sum_{l=1}^{q_t} \chi \big(V_l^{g_t} \cap f_t^{-1}(\delta) \cap B_{\varepsilon} 
\big)  \cdot {\rm Eu}_{X_t}(V_{l}^{g_t}),
\end{aligned}
\end{equation*}
where $V_i^{g_t}$ equals $V_i \cap g_t^{-1}(0)$ and $V_l^{g_t}$ denotes the strata of $\mathcal{V}_{f_t}$ which intersect $\Sigma_{\mathcal {V}_{f_t}}g_t$.

On the other hand, since the strata of $\mathcal{V}_{g_t}$ are the strata of $\mathcal{V}_{f_t}$ intersected by $\{g_t=0\},$
\begin{equation*}
\begin{aligned}
B_{f_t, X_t^{g_t}} (0)
&  =  \sum_{V_i \not\subseteq \Sigma_{\mathcal 
{V}_{f_t}}g_t}  \chi \big(V_i^{g_t} \cap f_t^{-1}(\delta) \cap B_{\varepsilon} 
\big)\cdot {\rm{Eu_{{X_t^g}}}}(V_i^{g_t}) \\ 
& + \sum_{l=1}^{q_t} \chi \big(V_l^{g_t} \cap f_t^{-1}(\delta) \cap B_{\varepsilon} 
\big)  \cdot {\rm Eu}_{X_t^{g_t}}(V_{l}^{g_t}).
\end{aligned}
\end{equation*}

Therefore,
\begin{eqnarray*}
 \sum_{V_i \subset{\mathcal {V}_{f_t}}}  \chi 
\big(V_i \cap X_t^{g_t} \cap f_t^{-1}(\delta) \cap B_{\varepsilon} \big) \cdot 
{\rm{Eu_{{X_t}}}}(V_i )& =& B_{f_t, X_t^{g_t}} (0) - \\
&-& \displaystyle 
 \sum_{l=1}^{r_t}\chi \big(V_l^{g_t} \cap f_t^{-1}(\delta) \cap B_{\varepsilon} 
\big) \cdot \rm{Eu_{X_t^{g_t}}}(V_l^{g_t})\\
&+& \sum_{l=1}^{q_t} \chi \big(V_l^{g_t} \cap f_t^{-1}(\delta) \cap B_{\varepsilon} 
\big)  \cdot {\rm Eu}_{X_t}(V_{l}^{g_t}).
\end{eqnarray*}
\end{proof}

Given $X \subset \mathbb{C}^n$ a Newton non-degenerate complete intersection and a generic linear form  with respect to $X$, $h: (\mathbb{C}^m,0) \to (\mathbb{C},0)$, the restriction of $h$ to $X$ may be degenerate if we eliminate one variable using $h=0$ (see \cite[Example (I-2)]{Oka}).
However, the variety $X \cap h^{-1}(0) \subset \mathbb{C}^n$ is Newton non-degenerate. Then, we have the following.

\begin{corollary}\label{corollary39}

Let $X_t \subset \mathbb{C}^n$ be the germ
given by $V(f_t^1, \dots, f_t^{k_0 -2})$ and $f_t=f_t^{k_0 -1}$ and $ g_t=f_t^{k_0}$ be non-constant polynomial functions on $X_t$. If the family $\{p_t^1\}_t$  is Newton-admissible and $f_t$ is a generic linear form with respect to $X_t$ then, for $0 < \left| \delta \right| 
	\ll \varepsilon \ll 1$,
\begin{equation}\label{morsepoints2}
   {\rm Eu}_{X_t}(0)-{\rm Eu}_{X_t^{g_t}}(0)- \displaystyle 
 \sum_{i=l}^{q_t}\chi \big(V_l^{g_t} \cap f^{-1}_t(\delta) \cap B_{\varepsilon} 
\big)   \big(\rm{Eu_{X_t^{g_t}}}(V_l^{g_t}) - 
 \rm{Eu_{X_t}}(V_l^{g_t})  \big)=(-1)^{d-1}m_t
\end{equation}
where $m_t$ is the number of stratified Morse critical points of a partial 
morsefication of $g_t: X_t\cap f_t^{-1}(\delta) \cap B_{\varepsilon} \to 
\mathbb{C}$ appearing on $({X_t})_{\rm reg} \cap f_t^{-1}(\delta) \cap 
 \{g_t 
\neq 0\} \cap B_{\varepsilon}$. In terms of weighted Euler characteristics
this can be reformulated (for $0<|\delta|\ll\varepsilon\ll1$) as
\begin{equation*}
 {\rm Eu}_{X_t}(0)- {\rm Eu}_{X_t^{g_t}}(0) -\chi(X_t^{g_t} \cap f_t^{-1}(\delta)\cap B_{\varepsilon}, {\rm Eu}_{X_t^{g_t}}-{\rm Eu}_{X_t})=(-1)^{d-1}m_t.
\end{equation*}
\end{corollary}

\begin{proof}
    Firstly, for a sufficiently generic function $f_t$, the set $V_{f_t}$ of all subsets $I \subseteq \{1, \dots , n\}$ satisfying $f_{t}|_{\mathbb{C}^I} \equiv 0$ is the empty set, i.e., $V_{f_t}= \emptyset$, since $\mathbb{C}^{\emptyset}$ is the origin of $\mathbb{C}^n$. Moreover, the construction presented in \cite[Section $6.3$]{EO2} guarantees that, for $I \notin V_{f_t},$ the Whitney's conditions depend only on the non-degeneracy of $f_t$. Then, since $\{p^1_t\}_t$ is Newton-admissible, we have that $\{p_t\}_t$ is Newton-admissible.
    
    Now, since $f_t$ is generic, the symmetric relative polar variety $\tilde{\Gamma}_{f_t,g_t}(\mathcal{V}_{f_t })$ has dimension less or equal to one. Moreover, by \cite[Proposition $5.2$]{EO2}, we may apply Lemma \ref{Masseyversion}.
\end{proof}

 We remark that the assumptions of non-degeneracy and
uniform local tameness are elementary algebraic conditions which can often
be checked using computational methods. Moreover, as we are going to see in the next section, using Matsui and Takeuchi results \cite{MT1} we can present formulae to compute the objects which appear on the left side of Equation (\ref{morsepoints}). Hence, we provide an algebraic approach to compute the number of Morse critical points (which are
geometric objects) .

\section{Euler obstruction, Morse points and torus action} 


In this section, we compute the local Euler obstruction and the Brasselet number of complete intersections varieties given by Newton-admissible families. As we said before, we use Matsui and Takeuchi results to state such formulae \cite{MT1}. 

We start presenting the definitions and notations from \cite{MT1} adapted to the special case in which  $X$ is a non-degenerate complete intersection in $\mathbb{C}^n$ (see also \cite{Oka}). However, we notice that Matsui and Takeuchi results hold for non-degenerate complete intersection in any affine toric variety. 

 Let $\mathbb{R}^n_{+}$ be the positive orthant of $\mathbb{R}^n$ and consider the following subvarieties 
$$X^f:= \{f^1=\dots=f^{k-2} = f^{k-1} =0 \} \subset X:= 
\{f^1=\dots=f^{k-2} =0 \}.$$ Assume that $0 \in X^f$. Since $\mathbb
C^n$ is the toric variety associated to the polyhedron cone generated by the canonical base of $\mathbb{R}^n$, that is, the positive orthant $\mathbb{R}^n_{+}$, in the following, we denote by $\Delta\prec\mathbb{R}^n_{+}$ a face of $\mathbb{R}^n_{+}.$   For each face 
$\Delta\prec\mathbb{R}^n_{+}$ such that $\Gamma_{+}(f^{k-1}) 
\cap \Delta \neq \emptyset$, we set 
\[
I_0^2(\Delta)= \bigsetdef{ j=1,2,\dots, k-2}{
\Gamma_{+}(f^j) \cap \Delta \neq \emptyset } \subset \left\{ 
1,2,\dots,k-2 \right\}
\] and $m_0^2(\Delta)= \sharp I_0^2(\Delta) + 1$, where $\sharp I_0^2(\Delta)$ denotes the cardinality of the set $ I_0^2(\Delta).$ 

Denoting the monomial $x_1^{v_1} \cdots x_{n}^{v_n}$ by $x^{v}$, in which $v = (v_1,\dots,v_n) \in \mathbb{Z}^n_{+}$, we have the following.

\begin{definition} 
	(i) For a polynomial function $f = \displaystyle{ \sum_{v\in \Gamma_{+}(f)}} a_v \cdot 
	x^v$
	on $\mathbb{C}^n$ and $u \in {\Delta}$, 
	we set 
 $$f|_{\Delta} = \sum_{v\in \Gamma_{+}(f) \cap \Delta} a_v \cdot x^v$$
 and $$\Gamma(f|_{\Delta};u) = 
	\left\{v \in \Gamma_{+}(f)\cap \Delta; \ \ \left\langle u, v\right\rangle 
	= {\rm min} \left\langle u,w\right\rangle, \ \ \text{for} \ \ w \in 
	\Gamma_{+}(f) \cap \Delta \right\}.$$ The set $\Gamma(f|_{\Delta};u)$ is called the 
	\textbf{supporting face} of $u$ in $\Gamma_{+}(f) \cap \Delta$.
	
	(ii) For $j \in I_0^2(\Delta) \cup \left\{k -1  \right\}$ and $u \in 
	{\Delta}$, we define the \textbf{$u$-part} $f_{u}^{j}$ of $f^j$ by 
	\[
	f_{u}^{j} = 
	\displaystyle{\sum_{v\in \Gamma(f^{j}|_{\Delta};u)} a_v \cdot x^v},
	\] 
	where $f^j = 
	\displaystyle{\sum_{v\in \Gamma_{+}(f^j)} a_v \cdot x^v}$.
\end{definition}

For each face $\Delta$ in $\mathbb{R}^n_{+}$ of $\mathbb{R}^n_{+}$ such that 
$\Gamma_{+}(f^{k-1}) \cap \Delta \neq \emptyset$, let us set 
$$
p_{\Delta} = 
\big(\displaystyle{\prod_{j\in I_0^2(\Delta)}} f^j\big)\cdot f^{k-1}
$$ 
and consider its Newton polygon $\Gamma_{+}(p_{\Delta}) = \left\{\sum_{j \in 
	I_0^2(\Delta)} \Gamma_{+}(f^j) \right\} + \Gamma_{+}(f^{k-1}) \subset 
\mathbb{R}^n_{+}$. Let $\gamma_{1}^{\Delta},\dots, 
\gamma_{\nu_0^2(\Delta)}^{\Delta}$ be the compact faces of 
$\Gamma_{+}(p_{\Delta}) \cap \Delta (\neq \emptyset)$ such that 
$\dim\gamma_{i}^{\Delta} = \dim\Delta -1$. Then, for each $1\leq i \leq 
\nu_0^2(\Delta)$, there exists a unique primitive vector $u_{i}^{\Delta} \in {\rm 
	Int}({\Delta})$ which takes its 
minimal in $\Gamma_{+}(p_{\Delta}) \cap \Delta$ exactly on 
$\gamma_{i}^{\Delta}$.

For $j \in I_0^2(\Delta) \cup \left\{k-1 \right\}$, set $\gamma(f^j)_{i}^{\Delta} 
:= \Gamma(f^{j}|_{\Delta};u_{i}^{\Delta})$ and $(d_0^2)_{i}^{\Delta}:= {\rm 
	min}_{w\in \Gamma_{+}(f^{k-1}) \cap \Delta} \left\langle 
u_{i}^{\Delta},w\right\rangle$. Note that we have $$\gamma_{i}^{\Delta} = 
\displaystyle{\sum_{j \in I_0^2(\Delta) \cup \left\{k-1 \right\}}} \gamma 
(f^j)_{i}^{\Delta}$$ for any face $\Delta$ in $\mathbb{R}^n_{+}$ satisfying
$\Gamma_{+}(f^{k-1}) \cap \Delta \neq \emptyset$ and $1 \leq i \leq \nu_0^2(\Delta)$. 
For each face $\Delta$ in $\mathbb{R}^n_{+}$ such that $\Gamma_{+}(f^{k-1}) \cap 
\Delta \neq \emptyset$, $\dim\Delta \geq {\rm m_0^2}(\Delta)$ and $1 \leq i 
\leq \nu_0^2(\Delta)$, we set $I_0^2(\Delta) \cup \left\{k-1\right\} = 
\left\{j_1,j_2,\dots,j_{{\rm m_0^2}(\Delta)-1},k-1=j_{{\rm m_0^2}(\Delta)} \right\}$ 
and $$(K_0^2)_{i}^{\Delta} := \sum_{{{ \alpha_1+\dots+\alpha_{{\rm m_0^2}(\Delta)}  =  
			\dim\Delta -1} \atop {\alpha_q \geq 1 \ \ \text{for} \ \ q\leq 
			{\rm 
				m_0^2}(\Delta)-1}} \atop {\alpha_{{\rm m_0^2}(\Delta)} \geq 0}}^{} {\rm 
	Vol}_{\mathbb{Z}}(\underbrace{\gamma(f^{j_1})_{i}^{\Delta},\dots,\gamma(f^{j_1})_{i}^{\Delta}}_{\alpha_1-\text{times}},\dots,\underbrace{\gamma(f^{j_{{\rm
				m_0^2}(\Delta)}})_{i}^{\Delta},\dots,\gamma(f^{j_{{\rm 
				m_0^2}(\Delta)}})_{i}^{\Delta}}_{\alpha_{{\rm 
				m}(\Delta)}-\text{times}}).$$ Here,
$${\rm 
	Vol}_{\mathbb{Z}}(\underbrace{\gamma(f^{j_1})_{i}^{\Delta},\dots,\gamma(f^{j_1})_{i}^{\Delta}}_{\alpha_1-\text{times}},\dots,\underbrace{\gamma(f^{j_{{\rm
				m_0^2}(\Delta)}})_{i}^{\Delta},\dots,\gamma(f^{j_{{\rm 
				m_0^2}(\Delta)}})_{i}^{\Delta}}_{\alpha_{{\rm 
				m}(\Delta)}-\text{times}})$$ is the 
normalized $(\dim\Delta -1)$-dimensional mixed volume with respect to 
the lattice $(\mathbb{Z}^n \cap \Delta) \cap L(\gamma_{i}^{\Delta})$, in which $L(\gamma_{i}^{\Delta})$ is the smallest linear subspace of $\mathbb{R}^n$ containing $\gamma_{i}^{\Delta}$ (see 
Definition $2.6$, pg $205$ from \cite{GKZ}). For $\Delta$ such that 
$\dim\Delta=1$, we set $$(K_0^2)_{i}^{\Delta} = {\rm 
	Vol}_{\mathbb{Z}}(\underbrace{\gamma(f^{k-1})_{i}^{\Delta},\dots,\gamma(f^{k-1})_{i}^{\Delta}}_{0-\text{times}})
:= 1 $$ (in this case $\gamma(f^{k-1})_{i}^{\Delta}$ is a point).

Now, let $X_t \subset \mathbb{C}^n$ be the germ
given by $V(f_t^1, \dots, f_t^{k_0 -2})$ and $f_t=f_t^{k_0 -1}, g_t=f_t^{k_0}$ be non-constant polynomial functions on $X_t$. With the previous notations, Matsui and Takeuchi prove the following \cite[last statement of Theorem $3.12$]{MT1}.

\begin{theorem}\label{last statement of ThMT}
   The Euler characteristic of the Milnor fiber of $f_t=f_t^{k_0-1}$ at
$0\in f_t^{-1}(0)$ is given by
$$\sum_{{\Gamma_{+} 
			(f_t) \cap \Delta \neq \emptyset} \atop {\dim \Delta \ \geq \ m_0^2(\Delta)}} 
			(-1)^{\dim \Delta \ - \ m_0^2(\Delta)} \left( \sum_{i=1}^{\nu_0^2(\Delta)} 
			(d^2_0)_{i}^{\Delta} \cdot (K_0^2)_{i}^{\Delta} \right).$$
\end{theorem}

If the family $\{p_t\}_t$  is Newton-admissible, the stratifications $\mathcal{V}_{f_t}$ of $X_t$ and $\mathcal{V}_{g_t}$ of $X_{t}^{g_t}$ satisfy Whitney’s conditions. Therefore, by Theorem
\ref{last statement of ThMT}
\begin{equation}\label{NB2}
			{\rm B}_{f_t,X_{t}}(0)   =  \sum_{{\Gamma_{+} 
			(f_t) \cap \Delta \neq \emptyset} \atop {\dim \Delta \ \geq \ m_0^2(\Delta)}} 
			(-1)^{\dim \Delta \ - \ m_0^2(\Delta)} \left( \sum_{i=1}^{\nu_0^2(\Delta)} 
			(d^2_0)_{i}^{\Delta} \cdot (K_0^2)_{i}^{\Delta} \right)\cdot 
			\rm{Eu}_{X_{t}}(T_{\Delta} ),
	\end{equation}
in which $T_{\Delta} = S^{I_{\Delta}}(K_0^2) \cap (\{t\}\times\mathbb{C}^n)\setminus\{f=0\}$ and $\mathbb{C}^{I_{\Delta}}$ is the subspace of $\mathbb{C}^n$ corresponding with the face $\Delta$ of $\mathbb{R}^n_{+}$.

\begin{remark}
    As the family $\{p_t\}_t$ is Newton-admissible, for any $k \in K_0$, the Newton boundary $\Gamma(f^k_t)$ does not depend on $t$. Then, in Equation (\ref{NB2}), $m_0^2(\Delta)$, $\nu_0^2(\Delta)$, $(d_0^2)_i^{\Delta}$ and $(K_0^2)_i^{\Delta}$ do not depend on $t$. 
\end{remark}

As before, for each face 
$\Delta$ in $\mathbb{R}^n_{+}$ such that $\Gamma_{+}(f^{k_0 -1}_t) 
\cap \Delta \neq \emptyset$, we set 
\[
I_0^{1}(\Delta)= \bigsetdef{ j=1,2,\dots, k-2, k_0}{
\Gamma_{+}(f^j_t) \cap \Delta \neq \emptyset } \subset \left\{ 
1,2,\dots,k-2, k_0 \right\}
\] and $m_0^1(\Delta)= \sharp I_0^1(\Delta) + 1$. Applying the Theorem 
\ref{last statement of ThMT} again, we have
	\begin{equation}\label{NB1}
			{\rm B}_{f_t,X_{t}^{g_t}}(0)   =  \sum_{{\Gamma_{+} 
			\big(f_t^{k_0-1}\big) \cap \Delta \neq \emptyset} \atop {\dim \Delta \ \geq \ m_0^1(\Delta)}} 
			(-1)^{\dim \Delta \ - \ m_0^1(\Delta)} \left( \sum_{i=1}^{\nu_0^{1}(\Delta)} 
			(d^1_0)_{i}^{\Delta} \cdot (K_0^1)_{i}^{\Delta} \right)\cdot 
			\rm{Eu}_{X_{t}^{g_t}}(T_{\Delta}^{g_t} ),
	\end{equation}
in which $T_{\Delta}^{g_t}= S^{I_{\Delta}}(K_0^1) \cap (\{t\}\times\mathbb{C}^n)\setminus\{f=0\}$ and $\mathbb{C}^{I_{\Delta}}$ is the subspace of $\mathbb{C}^n$ corresponding with the face $\Delta$ of $\mathbb{R}^n_{+}$ (there exists a natural action from the algebraic torus $T = (\mathbb{C}^{*})^n$ to $\mathbb{C}^n$. Moreover, the $T$-orbits of this action 
are in a 1-1 relation with the faces $\Delta$ of $\mathbb{R}^n_{+}$).

Similarly to the Equation (\ref{NB2}), in Equation (\ref{NB1}), $m_0^1(\Delta)$, $\nu_0^1(\Delta)$, $(d_0^1)_i^{\Delta}$ and $(K_0^1)_i^{\Delta}$ do not depend on $t$.

Moreover, as we said before, given a Newton non-degenerate complete intersection $X \subset \mathbb{C}^n$ and a generic linear form  with respect to $X$, $h: (\mathbb{C}^m,0) \to (\mathbb{C},0)$, the variety $X \cap h^{-1}(0) \subset \mathbb{C}^n$ is Newton non-degenerate. Therefore, using Theorems \ref{BLS} and \ref{last statement of ThMT}, we can also compute $\rm{Eu}_{X_{t}}(T_{\Delta} )$ and $\rm{Eu}_{X_{t}^{g_t}}(T_{\Delta}^{g_t})$ in terms of volumes of Newton polyhedra. Furthermore, these numbers do not depend on $t$ either. Then, from Theorem  \ref{diferencadenumerosdebrasselet1}, we have the following.

\begin{corollary}\label{corollary43}
Let $X_t \subset \mathbb{C}^n$ be the germ
given by $V(f_t^1, \dots, f_t^{k_0 -2})$ and $f_t=f_t^{k_0 -1}$ and $g_t=f_t^{k_0}$ be non-constant polynomial functions on $X_t$. If the family $\{p_t\}_t$ is Newton-admissible, then the number of stratified Morse critical points $m_t$ does not depend on $t$, for $t$ small enough. The same is true for the
Brasselet numbers $B_{f_t,X_t}(0)$ and $B_{f_t,X_t^{g_t}}(0)$.
\end{corollary}

\begin{proof}
    The constancy of the numbers $\chi \big(V_l^{g_t} \cap f^{-1}_t(\delta) \cap B_{\varepsilon} 
\big)$ can be obtained exactly as we did above for the Brasselet numbers $ B_{f_t,X_t}(0)$, $ B_{f_t,X_t^{g_t}}(0)$ 
 and for the local Euler obstructions $   \rm{Eu_{{X_t}^{g_t}}}(V_l^{g_t})$, $
 \rm{Eu_{X_t}}(V_l^{g_t}) $, since we can also apply Theorem \ref{last statement of ThMT}.
\end{proof}


\section*{Acknowledgements}
Part of this work was developed at IMPA (Brazil) through the project Thematic Program on Singularities, providing useful discussions with the authors. Moreover, the authors were partially supported by this project to which the authors are grateful.

\section*{Conflict of interests statement}

On behalf of all authors, the corresponding author states that there is no conflict of interest.

\end{document}